\documentclass{amsart}
\usepackage{amsthm, amsmath, amssymb}

\newtheorem{theorem}{Theorem}[section]
\newtheorem{proposition}[theorem]{Proposition}
\newtheorem{lemma}[theorem]{Lemma}
\newtheorem{corollary}[theorem]{Corollary}

\theoremstyle{definition}
\newtheorem{definition}[theorem]{Definition}
\newtheorem{example}[theorem]{Example}

\newcommand{\Dav}{\mathsf{D}}
\newcommand{\N}{\mathbb{N}}
\newcommand{\Z}{\mathbb{Z}}
\newcommand{\rank}{\mathsf{r}}
\newcommand{\fc}{\mathcal{F}}

\begin{document}

\title{Remarks on the plus-minus weighted Davenport constant}

\author{Luz E. Marchan \and Oscar Ordaz \and Wolfgang A. Schmid}

\address{Departamento de Matem\'aticas, Decanato de Ciencias y Tecnolog\'{i}as, Universidad Centroccidental Lisandro Alvarado, Barquisimeto, Venezuela}
\address{Escuela de Matem\' aticas y Laboratorio MoST, Centro ISYS, Facultad de Ciencias,
Universidad Central de Venezuela, Ap. 47567, Caracas 1041--A, Venezuela}
\address{Universit\'e Paris 13, Sorbonne Paris Cit\'e, LAGA, CNRS, UMR 7539, Universit\'e Paris 8, F-93430, Villetaneuse, France}

\email{luzelimarchan@gmail.com}
\email{oscarordaz55@gmail.com}
\email{schmid@math.univ-paris13.fr}

\thanks{The research of O. Ordaz is supported by Banco Central de Venezuela and Postgrado de la Facultad de Ciencias de la U.C.V. and CDCH project number 03-8018-2011-1; the one of W.A. Schmid by the PHC Amadeus 2012 project number 27155TH and the ANR project Caesar, project number ANR-12-BS01-0011.}

\begin{abstract}
For $(G,+)$ a finite abelian group the plus-minus weighted Davenport constant, denoted $\mathsf{D}_{\pm}(G)$, is the smallest $\ell$ such that each sequence $g_1 \dots g_{\ell}$ over $G$ has a weighted zero-subsum with weights $+1$ and $-1$, i.e., there is a non-empty subset $I \subset \{1,\dots, \ell\}$ such that $\sum_{i \in I} a_i g_i =0$ for $a_i \in \{+1,-1\}$. We present new bounds for this constant, mainly lower bounds, and also obtain the exact value of this constant for various additional types of groups.
\end{abstract}

\maketitle

\section{Introduction}

A zero-sum problem over a finite abelian group $G$ typically asks for the smallest value $\ell$ such that each sequence $g_1 \dots g_{\ell}$ of elements of $G$ has a (non-empty) subsequence whose terms sum to $0$ (and that possibly fulfills some additional condition), i.e., there is some non-empty subset $I \subset \{1, \dots, \ell\}$ such that $\sum_{i\in I} g_i = 0$ (and there might be additional conditions, e.g., on the cardinality of $I$). In case no additional condition is imposed this $\ell$ is called the Davenport constant of $G$, denoted by $\Dav(G)$. For a general overview on these types of problems see the survey article by Gao and Geroldinger \cite{gaogersurvey}.

There are different ways to generalize or modify these problems introducing weights. A particularly popular one introduced by Adhikari et al. \cite{adetal,adhi0} is to fix a certain set of weights $A$, a subset of the integers, and to ask for the existence of an $A$-weighted zero-subsum, i.e., $\sum_{i \in I}a_ig_i = 0$ for some non-empty subset $I \subset \{1, \dots, \ell\}$ and $a_i \in A$. The resulting constant is then called the $A$-weighted Davenport constant and is denoted by $\Dav_A(G)$.
That is, there is some specified set of admissible weights and one is free to use each weight as often as desired. There are also investigation in the situation where the  multiplicities for each weight are prescribed, see, e.g., \cite[Section 9]{gaogersurvey} for an overview; yet we do not consider this version of the problem.

In earlier literature, frequently, the constant $\mathsf{E}_A(G)$, defined in the same way except that the subset $I$ has to have cardinality $|G|$, was considered as well, yet by \cite{oscar-weighted-projectII} one knows that $\mathsf{E}_A(G) = \Dav_{A}(G) + |G| - 1$ so that the study of this invariant can be subsumed in the study of the weighted Davenport constant.

One particularly popular set to consider as sets of weights is the set $\{+1,-1\}$. This is not only due to the fact that it seems like a natural choice, but also since there are some applications. On the one hand, sequences or sets (but here this hardly makes a difference) without plus-minus weighted zero-subsum, they are called dissociated, play a role in other parts of additive combinatorics, in particular in certain Fourier-analytic arguments (see, e.g., \cite[Section 4.5]{taobook}). Moreover, the investigation of dissociated sets in the integers and in lattices is a classical problem as well; see for example \cite{levyuster} for a recent contribution. On the other hand, very recently Halter-Koch \cite{halterkoch} showed that the plus-minus weighted Davenport arises in questions on the norms of principal ideals in quadratic algebraic number fields---the relevant group is the class group and the presence of weights arises via the natural action of the Galois group on the class group---as well as questions related to binary quadratic forms.

In the present paper we focus almost exclusively on this set of weights.
We refer to $\Dav_{\{+1,-1\}}(G)$ as the plus-minus weighted Davenport constant and we denote it by $\Dav_{\pm}(G)$; this is mainly done for brevity of notation, yet also to hint at the fact that the results are not strictly limited to this set of weights, see Proposition \ref{prel_prop_gcd} and the subsequent discussion for details.

In a recent paper Adhikari, Grynkiewicz, and Sun \cite{adhikari-david} established the  following bounds for $\Dav_{\pm}(G)$. Let $G$ be isomorphic to $C_{n_1} \oplus \dots \oplus C_{n_r}$ with $n_i \mid n_{i+1}$ where $C_{n_i}$ denotes a cyclic group of order $n_i$. Then,
\begin{equation}
\label{eq_AGS}
\sum_{i=1}^r \lfloor \log_2 (n_i) \rfloor  +1 \le \Dav_{\pm}(G) \le \lfloor \log_2 |G| \rfloor +1.
\end{equation}

In special cases, such as the case that $G$ is cyclic or $G$ is $2$-group, the upper and the lower bound coincide and  thus yield the exact value of the plus-minus weighted Davenport constant.
A natural problem is to try to understand how the constant $\Dav_{\pm}(G)$ behaves in cases where the upper and lower bounds do not match.
To begin a detailed investigation of this question, which to the best of our knowledge so far has not been considered, is the main theme of the present paper.

On the one hand, we show two ways, quite different in nature, how the lower bound can be improved for certain groups.
To the best of our knowledge these are the first examples of such improvements.

The problem of improving this lower bound is somehow analogous to the problem of finding groups where the Davenport constant is larger then the lower bound given by $\Dav^{\ast}(G) = \sum_{i=1}^r (n_i-1) + 1$---note that both this lower bound and the one in \eqref{eq_AGS} are obtained via passing to cyclic groups where the exact value is known---for a recent contribution to that problem see \cite{glp}. Yet, in fact, we will argue that the lower bound in \eqref{eq_AGS} is not yet the `true analog' of $\Dav^{\ast}(G)$; and identifying the `true analog' of $\Dav^{\ast}(G)$ is the first way how we improve the lower bound. It goes without saying that `true analog' is not a precise notion and this assertion thus cannot be a definite statement, however in Section \ref{sec_revisited} there is some detailed justification for our point of view on this matter.
Having done this, we proceed to further explore this problem and give a construction that does not only consider the cyclic subgroups, and that yields further improvements in certain cases (see Section \ref{sec_MLB}).

On the other hand, we also give a first example of an improvement of the upper bound for groups other than elementary $3$-groups (for the relevance of this remark see Theorem \ref{thm_el3} and the subsequent discussion).

As a by-product of the improvement of the bounds (in particular the lower bound) we obtain the exact value of the plus-minus weighted Davenport constant for various new families of groups. In Sections \ref{sec_revisited} and \ref{sec_MLB} we focus on the methods, and some of the results are thus quite technical; yet, typically, we complement them with  some applications to specific types of groups. In addition, the final Section \ref{sec_rem} contains various explicit results, among others obtaining the exact value of the plus-minus weighted Davenport constant for all but one groups of order at most $100$.

\section{Preliminaries and notation}
\label{sec_prel}

In this section we gather some notation as well as some basic facts.
As usual we denote for a real number $x$ by $\lfloor x \rfloor = \max \{ z\in \Z \colon z \le x\}$ the floor of $x$ and we denote by $\{x\} = x - \lfloor x \rfloor$  its fractional part. By $\log_2 x$ we denote the logarithm in base $2$ of $x$.

We use additive notation for abelian groups.
For $n$ a positive integer, we denote by $C_n$ a cyclic group of order $n$.
For $G$ a nontrivial finite abelian group, where $1< n_1 \mid \dots \mid n_r$ such that
\[G \cong C_{n_1} \oplus \dots \oplus C_{n_r},\]
we denote $\exp(G)=n_r$ the exponent of $G$, $\rank(G)=r$ the rank of $G$ and for $n \in \N$ by $\rank_n(G)= |\{i \colon n \mid n_i  \}|$ the $n$-rank of $G$, and finally by $\rank^{\ast}(G) = \sum_{p \in \mathbb{P}} \rank_p (G)$, where $\mathbb{P}$ denotes the set of prime numbers, the total rank of $G$. A group is called a $p$-group if its exponent is a prime power and it is called elementary if its exponent is square-free.

A finite family of non-zero elements $(e_1, \dots, e_s)$  of $G$ is called independent if
$\sum_{i=1}^s a_i e_i = 0$ with $a_i \in \Z$ implies that $a_ie_i = 0$ for each $i$. An independent family whose elements are a generating set of $G$ is called a basis of $G$. For $(e_1, \dots, e_s)$ a basis of $G$, we know that $G = \oplus_{i=1}^s \langle e_i \rangle$. Thus, each element $g\in G$, can be written in a unique way as $g= g_1 + \dots + g_s$ where $g_i \in \langle e_i \rangle$, and of course each $g_i$ is equal to $a_ie_i$ for some integer $a_i$ uniquely determined modulo the order of $e_i$. We do not impose some global convention on the interval in which the $a_i$ are chosen, since at different occasion different conventions will be useful, and we make them clear locally.
  When a basis is fixed (explicitly or implicitly), we refer to $g_i=a_ie_i$ as $e_i$-coordinate of $g$.

For a subset $M$ of an (additive) abelian group, possibly the full group or a subset of the integers, and an integer $b$ we denote by $bM= \{bm \colon m \in M\}$ the dilation of $M$ by $b$; not the $b$-fold sumset of $M$.

Next, we recall some terminology related to sequences (in the way we use this term, which is common in the present context).
Let $G$ be a finite abelian group, written additively.
By a sequence over $G$ we mean an element of $\fc(G)$ the free abelian monoid over $G$; we use multiplicative notation for this monoid.
Let $S \in \fc(G)$. By definition this means there exist (up to ordering) uniquely determined $g_1, \dots, g_{\ell} \in G$ (some of them possibly equal) such that $S= g_1 \dots g_{\ell}$. The neutral element of $\fc(G)$, the empty sequence, would correspond to the empty product, i.e., $\ell = 0$.
A subsequence of $S$ is a divisor $T$ of $S$ in $\fc(G)$ ; in other words, a sequence of the form $\prod_{i\in I}g_i$ where $I \subset [1, \ell]$.

We call $\ell$ the length of the sequence and $\sigma(S) =  \sum_{i=1}^{\ell} g_i$ the sum of $S$. Moreover, $\Sigma (S) = \{ \sum_{i \in I}g_i \colon \emptyset \neq I  \subset [1,\ell]\} =  \{\sigma (T) \colon 1\neq T \mid S\}$ is called the set of subsums of $S$.

Let $A \subset \Z$ be a non-empty subset; if used in the context given below we refer to it as a set of weights.
We mainly are interested in the set
\[\Sigma_A (S) = \{ \sum_{i \in I}a_i g_i \colon \emptyset \neq I  \subset [1,\ell] , \, a_i \in A \},\]
called the set of $A$-weighted subsums of $S$; in \cite{oscar-weighted-projectII} an index-free definition of this set was given as well.
Sometimes we also need to consider just the elements arising from the sums corresponding to the sequence itself, i.e., $\sum_{i=1}^{\ell}a_i g_i$ with $a_i\in A$; we call such an element an $A$-weighted sum of $S$.

A map $f:G \to H$ between finite abelian groups yields in a natural way a map, indeed a monoid homomorphism, between $\mathcal{F}(G)$ and $\mathcal{F}(H)$ that we also denote by $f$. In case $f:G \to H$ is a group homomorphism, we have $\sigma(f(S))= f(\sigma(S))$ and $f(\Sigma_A(S))= \Sigma_A(f(S))$ for $S \in \mathcal{F}(G)$ and $A\subset \mathbb{Z}$ a set of weights.

We now recall more formally the definition of $A$-weighted Davenport constants.
\begin{definition}
Let $G$ be a finite abelian group and let $A\subset \mathbb{Z}$.
The $A$-weighted Davenport constant of $G$, denoted $\Dav_A(G)$, is the smallest $\ell \in \N$ such that for each $S \in \fc(G)$ with $|S| \ge \ell$ one has
\(0 \in \Sigma_A(S),\)
or in other words the smallest $\ell$ such that each sequence of that length has a non-empty subsequence $T$ such that $0$ is an $A$-weighted sum of $T$.
\end{definition}

Of course, the case $A= \{1\}$ yields the classical version of the Davenport constant.
It is immediate that for $n= \exp(G)$ and $\pi: \Z \to \Z/n \Z$ the standard map,  for $A, A' \subset \Z$ one has: if $\pi (A) = \pi(A')$, then
$\Sigma_A(S)= \Sigma_{A'}(S)$ and thus in particular
$\Dav_A(G)= \Dav_{A'}(G)$.

Moreover, for $b \in \Z$ one has that
$\Sigma_{b  A}(S)= b \Sigma_{A}(S)$.
Thus, if $b$ is coprime to $n$, since then multiplication by $b$ is an automorphism of $G$, one has $\Dav_A(G)= \Dav_{b  A}(G)$.
To extend this result to multiplication by any integer requires slightly more care; we give the main part of the argument in the following lemma.

\begin{lemma}\label{prel_lem_gcd}
Let $G$ be a finite abelian group with exponent $n$. Let $A \subset \Z$ be a non-empty set, and $d$ be a divisor of $\exp(G)$.
Then $\Dav_{d  A}( G ) = \Dav_{A}(dG)$.
\end{lemma}
\begin{proof}
Let $m_d: G \to G$ denote the map defined by $g \mapsto dg$. This is a group homomorphism and its image is $dG$.

First, let $S\in \fc(dG)$ of length $\Dav_{d  A}(G)$.
We need to show that $0 \in \Sigma_{A}(S)$.
Let $S'\in \fc(G)$ be a sequence such that $m_d(S')=S$.
By definition of $\Dav_{dA}(G)$ we get that there exist a subsequence $g_1'\dots g_t'\mid S'$ and sequence of weights $a_1', \dots, a_t'\in dA$ such
$\sum_{i=1}^t a_i' g_i' = 0$. By definition of $dA$ we know that each $a_i'$ equals $da_i$ for some $a_i \in A$.
Since
\[\sum_{i=1}^t a_i' g_i'  = \sum_{i=1}^t da_i g_i' = \sum_{i=1}^t a_i m_d(g_i') \]
and since $m_d(g_1') \dots m_d(g_t')$ is a subsequence of $S$, which is equal to $m_d(S')$ by definition of $S'$, we have shown that $0 \in \Sigma_{A}(S)$.

Now, conversely, let  $S'\in \fc(G)$ of length $\Dav_{ A}(dG)$.
The sequence $S= m_d(S')$ is a sequence over $dG$ and by the condition on its length we know that there is a subsequence $g_1\dots g_t \mid S$ and weights $a_i, \dots, a_t \in A$ such that
\[\sum_{i=1}^t a_i g_i = 0.\]
Let $g_1' \dots g_t'\mid S'$ be a subsequence such that $m_d(g_i') = g_i$ for every $i$.
Setting $a_i'= da_i$, we have
\[\sum_{i=1}^t a_i' g_i' = \sum_{i=1}^t a_ig_i =0, \]
and thus $0 \in \Sigma_{d A}(S')$, completing the argument.
\end{proof}

Combining the two preceding assertions we get the following result.

\begin{proposition}\label{prel_prop_gcd}
Let $G$ be a finite abelian group and $A  \subset \Z$ a non-empty subset, and $b$ an integer. Then
\[\Dav_{bA}(G) = \Dav_{A}(\gcd (b,\exp(G)) G ).\]
\end{proposition}
\begin{proof}
Let $d = \gcd(\exp(G), b)$ and write $b = d b'$. By the preceding lemma we get that $\Dav_{dA}(G) = \Dav_{A}(d G )$ and by Lemma \ref{prel_lem_gcd} applied with $b'$, which is coprime to the exponent of $G$, we get $\Dav_{b' (d A)}(G)= \Dav_{dA}(G)$, establishing the claim as of course $b'  (d  A) = (b'd) A$.
\end{proof}

As a consequence of this proposition it suffices (essentially) to consider sets of weights $A$ whose greatest common divisor is $1$. In more detail, suppose $A$ is a set of weights and let $d= \gcd(A)$.
Then writing $A = d  A'$, with $\gcd(A')=1$, we get that $\Dav_A(G) = \Dav_{A'}(\gcd (d,\exp(G)) G  )$. Thus, the problem of determining $\Dav_A(G)$ for any $A$ is equivalent to the problem of determining $\Dav_{A'}(H)$ for some set $A'$ such that $\gcd A' = 1$ and some other group $H$, and indeed this group $H$ is rather simpler than the original group. Consequently, the main problem is to understand the problem for sets of weights that are coprime.

Finally, we see that the investigations of the present paper focused on the set of weights $\{+1,-1\}$ are in fact slightly more general in that they apply to any set of weights $A$ whose image $\overline{A}$ in $\Z/ \exp(G) \Z $ is of the form $\{-g,g\}$ for some $g \in \Z/\exp(G) \Z$; note that this does include the case that $-g=g$ though it is not of much interest, as it merely leads to considering the (classical) Davenport constant of elementary $2$-groups.

We make frequent use of the following lower bound (see \cite[Lemm 3.1]{oscar-weighted-projectII}).

\begin{lemma}\label{prel_lem_lb}
Let $G$ be a finite abelian group and let $A \subset \Z$ be a non-empty set. Then \[\Dav_A(G) \ge \Dav_A(G/H) + \Dav_A(H) - 1.\]
\end{lemma}

In combination with the upper bound from \eqref{eq_AGS} one obtains the following results.

\begin{lemma}
\label{prel_lem_lb2}
Let $H_1,H_2$ be finite abelian groups such that
\[\Dav_{\pm}(H_i)= \lfloor \log_2 |H_i| \rfloor+1\]
and such that $\lbrace \log_2 |H_1| \rbrace +\lbrace \log_2 |H_2| \rbrace <1$.
Then for every finite abelian group $G$ that is an extension of $H_1$ by $H_2$ we have
\[\Dav_{\pm}(G) = \lfloor \log_2 |G| \rfloor + 1.\]
In particular, if $G$ has a subgroup $H$ such that $\Dav_{\pm}(H)= \lfloor \log_2 |H| \rfloor+1$ and $G/H$ is a $2$-group, then $\Dav_{\pm}(G)= \lfloor \log_2 |G| \rfloor+1$.
\end{lemma}
\begin{proof}
Let $G$ be a finite abelian group that is an extension of $H_1$ by $H_2$,
that is $G$ has a subgroup $H$ isomorphic to $H_2$ such that $H_1$ is isomorphic to $G/H$.
So, by Lemma \ref{prel_lem_lb} and assumption, we get that
\[\Dav_{\pm}(G) \ge \Dav_{\pm}(H_1) + \Dav_{\pm}(H_2) -1 = \lfloor \log_2 |H_1| \rfloor + \lfloor \log_2 |H_2| \rfloor +1.\]
Now, by assumption on the fractional parts
\[\lfloor \log_2 |H_1| \rfloor + \lfloor \log_2 |H_2| \rfloor = \lfloor \log_2 (|H_1| \ |H_2|) \rfloor , \]
and since of course $|H_1| \ |H_2| = |G|$ this lower bound matches the general upper  $\lfloor \log_2 |G| \rfloor + 1$, and implies the claimed equality.

The additional statement follows directly as $\{\log_2 |G/H|\}=0$ and $\Dav_{\pm}(G/H)= \lfloor \log_2 |G/H| \rfloor+1$ by \eqref{eq_AGS}.
\end{proof}

\section{The bounds of Adhikari--Grynkiewicz--Sun, revisited}
\label{sec_revisited}

In the present section we discuss \eqref{eq_AGS} and present a slight generalization of it. The main point here is not that this generalization can be obtained, this is hardly an insight, yet that it is relevant to consider this generalization.

\begin{theorem}
\label{thm_standard}
Let $G$ be a finite abelian group and let $m_1,\dots, m_t$ be positive integers such that $G$ is isomorphic to $\oplus_{i=1}^t C_{m_i}$. Then
\[\sum_{i=1}^t \lfloor \log_2 m_i  \rfloor +1 \le \Dav_{\pm} (G) \le \lfloor \sum_{i=1}^t \log_2 m_i \rfloor + 1. \]
\end{theorem}
Of course the upper bound is always $\lfloor \log_2 |G| \rfloor + 1$.
The difference to the above mentioned result is merely that we do not insist on the $m_i$ forming a chain of divisors.

\begin{proof}
To get the lower bound it suffices to apply Lemma \ref{prel_lem_lb} to reduce the problem to one for cyclic groups. Then one uses the well-known fact that $\Dav_{\pm}(C_m) \ge \lfloor \log_2 m  \rfloor +1$; as we need to refer to it later, we recall that $e(2e)\dots (2^{\lfloor \log_2 m  \rfloor - 1})e$, where $e$ is a generating element of $C_m$, can serve as example for a sequence of length $\lfloor \log_2 m  \rfloor$ without plus-minus weighted zero-subsum. The upper bound is in fact identical to the one in \eqref{eq_AGS}.
\end{proof}

Now, at first, it might not be clear why we insist on making this slight generalization.
For the Davenport constant (without weights) it is known that if one wishes to establish a lower bound via writing the group as a direct sum of cyclic groups and then plugging in the exact value for cyclic group, as we do here, then among all decompositions choosing the one where the orders form a chain of divisors is optimal. And, for the cross number it is known that to choose the orders to be prime powers is always optimal; the cross number $\mathsf{K}(G)$ of $G$ is the smallest value such that each sequence of that cross number has a zero-sum subsequence (the cross number of a sequence is the sum of the reciprocals of the orders of elements in the sequence); see, e.g., \cite[Section 9]{gaogersurvey} for details. Yet, for the plus-minus weighted Davenport constant this is not the case, and depending on precise circumstances different types of decompositions can be optimal. Before discussing this is more detail, we illustrate it by two explicit examples.

\begin{example} \label{ex_lb} \
\begin{enumerate}
\item We have $C_3 \oplus C_{3 \cdot 11 \cdot 23} \cong C_{3 \cdot 11} \oplus C_{3 \cdot 23}$. The former decomposition, the standard one, yields $(1+ 9) + 1=11$ as lower bound, whereas the latter yields $(5+6)+1=12$.
\item We have $C_{3 \cdot 13 \cdot 23}^2 \cong C_{3 \cdot 13} \oplus C_{3 \cdot 23} \oplus C_{13 \cdot 23}$. The former yields $(9 +9 ) + 1 = 19$ and the latter $(5 + 6 + 8)+1 = 20$.
\end{enumerate}
\end{example}

The underlying phenomenon is that one needs to choose a direct-sum decomposition such that the losses when taking the integral parts are minimal.
We make the following definition.

\begin{definition}
Let $G$ be a finite abelian group. Then
\[\Dav_{\pm}^{\ast}(G) = \max  \{\sum_{i=1}^t \lfloor \log_2 m_i  \rfloor +1 \colon G \cong \oplus_{i=1}^t C_{m_i}, \text{ with }  t, \, m_i \in \N  \}.\]
\end{definition}
For notational convenience this definition allows cyclic components of order one, which does not affect the definition in any substantial way.

Possibly our notation needs some explanation, in particular as it is slightly more than only some notation.
We choose the notation $\Dav_{\pm}^{\ast}(G)$ in analogy with the notations $\Dav^{\ast}(G)$ and $\mathsf{K}^{\ast}(G)$ for the Davenport constant and the cross number, respectively.
As recalled in the introduction $\Dav^{\ast}(G) = 1  + \sum_{i=1}^r (n_i -1)$ where $n_i$ are natural numbers with $n_i \mid n_{i+1}$ such that $G \cong \oplus_{i=1}^r C_{n_i}$ and $\mathsf{K}^{\ast}(G) = 1/\exp(G) + \sum_{i=1}^t(q_i-1)/q_i$ where $q_i$ are prime-powers such that $G \cong \oplus_{i=1}^t C_{q_i}$.

It seems to us that this definition is the correct analog of these definitions, in the sense that both $\Dav^{\ast}(G)$ and $\mathsf{K}^{\ast}(G)$ also would arise when considering the maximum over all lower bounds obtained via decomposing $G$ into cyclic components (and applying direct bounds for cyclic groups).
Of course, in those cases one does not define the constants via taking a maximum, since a decomposition is known that always yields the maximum (chain of divisors and prime-powers, respectively).

While it would be convenient to also know such a decomposition in our case, we believe that it is impossible (in a uniform way) to specify it.
This is already somewhat illustrate by the example above. We further elaborate on this point.

Let $G$ be a finite finite abelian group, and let $q_1, \dots, q_s$ denote prime powers (not $1$) such that $G \cong \oplus_{i=1}^s C_{q_i}$; of course the $q_i$ are uniquely determined up to ordering and $s$ is the total rank of $G$.
Furthermore, let $f_i = \{\log_2 q_i \}$.

Let $\mathcal{P}$ denote the set of primes dividing the order of $G$, and for each $p \in \mathcal{P}$ let $I_p$ denote the subset of $[1,s]$ formed by the $i$'s such that $q_i$ is a power of $p$.

One then notices that to find $\Dav_{\pm}^{\ast}(G)$ is equivalent to finding a decomposition $[1,s]= \cup_{i=1}^t J_i$ with pairwise disjoint $J_i$ such that $|J_i \cap I_p| \le 1$ for all $i,p$ such that
\[\sum_{i=1}^t \lfloor \sum_{j\in J_i} f_j \rfloor \]
is maximal (among all such decompositions of $[1,s]$).

There seems to be no general criterion to determine which decomposition is optimal, and this problem might even be difficult from an algorithmic point of view (though we did not explore this). At first glance, one might think that to minimize $t$ is a good idea, yet as documented by Example \ref{ex_lb}.2, this does not necessarily yield the optimal value. Of course, in case there are two sets $J_j$ and $J_k$ such that their union still has the property of having with each $I_p$ at most one element in common, then one can consider the decomposition obtained via replacing these two sets by their union without decreasing the sum in question. In other words, there is always a direct-sum decomposition attaining the maximum such that no two cyclic groups in it have coprime order. Yet, it is possible that the maximum is only attained for decompositions such that the orders of the cyclic groups (in their entirety) are coprime; in other words, the $t$ for each optimal decomposition can be greater than the rank of the group; see Example \ref{ex_lb}.2 for an explicit example.

For future reference we formulate the following corollary to Theorem \ref{thm_standard}.

\begin{corollary}
\label{cor_davpm}
Let $G$ be a finite abelian group.
Then
\[\Dav^{\ast}_{\pm} (G) \le \Dav_{\pm}(G) \le \Dav^{\ast}_{\pm} (G) + \rank(G)-1.\]
\end{corollary}
\begin{proof}
The first inequality is immediate by Theorem \ref{thm_standard} and the definition of $\Dav^{\ast}_{\pm} (G)$. For the second inequality it suffices to note that by definition we can write $G$ as the direct sum of $\rank(G)$ cyclic groups and for any choice of real numbers $x_1, \dots , x_t$ one has $\sum_{i=1}^t \lfloor x_i \rfloor \le  \lfloor \sum_{i=1}^t  x_i \rfloor + (t-1)$.
\end{proof}
Indeed, as the proof shows the upper bound in the corollary could be improved by replacing $\Dav^{\ast}_{\pm} (G)$ by the minimum over $\sum_{i=1}^r \lfloor \log_2 m_i \rfloor +1$ with $G \cong \oplus_{i=1}^{r} C_{m_i}$ and $\rank(G)=r$.

Moreover, the observations of this section already allow to improve on the known lower bounds of $\Dav_{\pm}(G)$ for certain $G$, namely those where the choice of $m_1, \dots, m_t$ as a chain of divisors does not attain the maximum, and also allows to obtain in some of these cases the exact value of $\Dav_{\pm}(G)$, namely when $\Dav_{\pm}^{\ast}(G)$ matches the upper bound $\lfloor \log_2 |G| \rfloor +1$.
Various explicit results could be formulated; we limit ourselves to two examples, in the spirit of Example \ref{ex_lb}.

\begin{theorem}
Let $n,m_1,m_2 \in \mathbb{N}$ pairwise coprime integers such that $\{\log_2 m_i\} + \{\log_2 n \} \ge 1$ for each $i$ and  $\{\log_2 m_1\}+ \{\log_2 m_2\}+ 2 \{\log_2 n \} < 3$. Then
\[\Dav^{\ast}_{\pm}(C_n \oplus C_{nm_1 m_2})= \Dav_{\pm}(C_n \oplus C_{nm_1 m_2}) = \lfloor \log_2 (n^2 m_1 m_2) \rfloor + 1.\]
\end{theorem}
\begin{proof}
Note that $C_n \oplus C_{nm_1 m_2} \cong C_{nm_1} \oplus C_{nm_2}$, and $\lfloor \log_2 (n m_1) \rfloor  + \lfloor \log_2 (n m_2) \rfloor = \lfloor \log_2 (n^2 m_1 m_2) \rfloor $, by the conditions on the fractional parts.
The claim follows by Theorem \ref{thm_standard}.
\end{proof}

\begin{theorem}
Let $n_1,n_2,n_3 \in \mathbb{N}$ pairwise coprime integers such that $\{\log_2 n_i\} + \{\log_2 n_j \} \ge 1$ for each $i\neq j$ and  $\{\log_2 n_1\}+ \{\log_2 n_2\}+ 2 \{\log_2 n_3 \} < 2$. Then
\[\Dav^{\ast}_{\pm}(C_{n_1n_2n_3}^2 )= \Dav_{\pm}(C_n \oplus C_{nm_1 m_2}) = \lfloor 2 \log_2 (n_1 n_2 n_3) \rfloor + 1.\]
\end{theorem}
\begin{proof}
We note that $C_{n_1n_2n_3}^2 \cong C_{n_1n_2} \oplus C_{n_1 n_3}  \oplus C_{n_2 n_3}$, and that $\lfloor \log_2 (n_1 n_2) \rfloor  + \lfloor \log_2 (n_1 n_3)  \rfloor   + \lfloor \log_2 (n_2 n_3)  \rfloor = \lfloor 2 \log_2 (n_1 n_2 m_3) \rfloor $, by the conditions on the fractional parts. Now, the claim follows by Theorem \ref{thm_standard}.
\end{proof}

Yet, of course $\Dav_{\pm}^{\ast}(G)$ does not always match the upper bound
$\lfloor \log_2 |G| \rfloor +1$.
Thus, a question that imposes itself is whether or not  there exist groups $G$ such that $\Dav_{\pm}(G)$ is indeed strictly larger than $\Dav_{\pm}^{\ast}(G)$. In view of the fact that this is so for the analogue question for the Davenport constant (see, e.g., \cite{glp}) but unknown, and in fact conjectured not to be so for the closely related cross number (see \cite[Conjecture 9.4]{gaogersurvey}), the answer a priori does not seem clear. In the following section, we see that such groups indeed exist.

\section{More lower bounds}
\label{sec_MLB}

The goal of this section is to exhibit some examples of groups where the actual value of $\Dav_{\pm}(G)$ does not coincide with $\Dav_{\pm}^{\ast}(G)$.
This will also allow to obtain the exact value of $\Dav_{\pm}(G)$ in some additional cases.

In the following result we present another construction for sequences without plus-minus weighted zero-subsum. We stress that in the result below we do not impose any divisibility or relative size condition on $m_1$ and $m_2$.

\begin{proposition}
\label{prop_constr}
Let $m_1,m_2$ be integers with $m_1\ge 4$ and $m_2 \ge 3$.
Then
\[\Dav_{\pm} (C_{m_1} \oplus C_{m_2}) \ge \lfloor \log_2 (m_1/3) \rfloor  +  \lfloor \log_2 (m_2/3)\rfloor + 4.\]
\end{proposition}

\begin{proof}
Let $e_1, e_2$ be generating elements of $C_{m_1} \oplus C_{m_2}$ such that the orders of $e_1$ and $e_2$ are $m_1$ and $m_2$, respectively; of course, $e_1$ and $e_2$ are then independent. We construct a sequence of length $\lfloor \log_2 (m_1/3) \rfloor  +  \lfloor \log_2 (m_2/3)\rfloor + 3$ without plus-minus weighted zero-subsum.

Let $k= \lfloor \log_2 (m_1/3) \rfloor$ and let $\ell = \lfloor \log_2 (m_2/3)\rfloor $.
Let $d = m  - 2^k$ where $m$ is the integer closest to $m_1/6$ (in case this is a half-integer we round up).

We consider the sequences $T_1$, $T_2$, and $T_3$ formed by the elements
\begin{itemize}
\item $2^i e_1$ for $i\in [0,k-1]$,
\item $2^j3 e_2$ for $j\in [0,\ell-1]$, and
\item $de_1+e_2$, $(d+2^k)e_1+e_2$, and $(d+2^{k+1})e_1+e_2$.
\end{itemize}

The length of the sequence $T_1T_2T_3$ is $ k +\ell +3$ and we need to show that it has no plus-minus weighted zero-subsum.

The sequence $T_1$, compare the proof of Theorem \ref{thm_standard}, has no plus-minus weighted zero-subsum and for similar reasons $T_2$ has no such subsum either; more precisely, the sets of plus-minus weighted subsums are
\[A_1=\{-(2^k-1)e_1 , \dots, -e_1,e_1, \dots, (2^k -1)e_1\}\]
 and
\[A_2=\{-3(2^{\ell}-1)e_2 , \dots, -3e_2,3e_2,6e_2 \dots, 3(2^{\ell} -1)e_2\},\]
respectively.
Thus, a zero-subsum of $T_1T_2T_3$ has to contain elements from $T_3$ and more precisely either two with opposite weights or all three with the same weight.

First assume the former. In this case the $e_1$-coordinate of the sum of these two elements, call it $a$,  is $\pm 2^ke_1$ or $\pm 2^{k+1}e_1$. Since by the choice of $k$ we have $3\ 2^k\le m_1$, in neither case $-a \in A_1$. Since all elements in $T_2$ have $e_1$-coordinate equal to $0$, it is thus impossible to have a plus-minus weighted zero-subsum of this form.

Second assume the later, and without restriction assume the three elements are weighted with plus. In this case the $e_1$-coordinate, call it $b$, of the sum of these three elements is one of
\[
\frac{m_1+\epsilon}{2}e_1 \text{ with } \epsilon \in \{-2,\dots,3\};\]
which one precisely depends on $m_1$. For none of these we have $-b \in A_1$; for $m_1\ge 9$ this follows directly noting
\[2^k \le \frac{m_1}{3} \le \frac{m_1-\epsilon}{2} \le \frac{2m_1}{3} \le m_1 -2^k,\]
and for the remaining values by the same argument just considering the actual values of $b$, and thus $\epsilon$, and $k$. So, again no plus-minus weighted zero-subsum is possible.
\end{proof}

We proceed to discuss the quality of this construction relative to the one discussed earlier (see Theorem \ref{thm_standard}), which yields a lower bound of
\[\lfloor \log_2 m_1 \rfloor  + \lfloor \log_2 m_2 \rfloor +1 \] for the plus-minus weighted Davenport constant of $C_{m_1} \oplus C_{m_2}$, assuming for simplicity that there is no other or at least no `better' decomposition as sum of cyclic groups, which is the case for example if both $m_1$ and $m_2$ are powers of the same prime number.

We have that
\[\lfloor \log_2 (m_1/3) \rfloor  +  \lfloor \log_2 (m_2/3)\rfloor + 4\]
equals
\[\lfloor \log_2 (m_1) \rfloor  +  \lfloor \log_2 (m_2)\rfloor + 1  + \delta\]
where
\begin{itemize}
\item $\delta = 1$ if $\{\log_2 m_i \} \ge \{ \log_2 3\}$ for both $i=1$ and $i=2$.
\item $\delta = 0$ if $\{\log_2 m_i \} \ge \{ \log_2 3\}$ for exactly one of  $i=1$ and $i=2$.
\item $\delta =-1$ if $\{\log_2 m_i \} \ge \{ \log_2 3\}$ for none of  $i=1$ and $i=2$.
\end{itemize}

Thus, we see that the construction given in Proposition \ref{prop_constr} under certain circumstances can yield an improvement over the known construction, yet under other circumstances it can yield a worse result.

We now present some consequences of this. We start with the following technical result and then apply it in some more restricted situations, where the improvement becomes more apparent.

\begin{proposition}
\label{prop_lb}
Let $m_1, \dots, m_r$ be positive integers with  $\{\log_2 m_i \} \ge \{\log_2 3 \}$ and $m_i \ge 4$ for at least $\lfloor r/2 \rfloor $ of them.
If $\sum_{i=1}^r \{\log_2 m_i \} < \lfloor r/2 \rfloor + 1$,
then
\[ \Dav_{\pm}( \oplus_{i=1}^r C_{m_i}) =  \lfloor \sum_{i=1}^r \log_2 m_i \rfloor +1.\]
\end{proposition}
\begin{proof}
We know that $\Dav_{\pm}( \oplus_{i=1}^r C_{m_i}) \le  \lfloor \sum_{i=1}^r \log_2 m_i \rfloor +1$ by Theorem \ref{thm_standard}.
So suppose all assumptions are fulfilled and we need to show that this is also a lower bound. Let $s = \lfloor r/2 \rfloor$. Without restriction we assume $m_1 \le \dots \le m_r$; note that $m_i$ is at least $3$ by the condition $\{\log_2 m_i \} \ge \{\log_2 3 \}$.
Moreover, we know that for each $i \in [r-s+1,r]$ we have $m_{i}\ge 4$.
Thus for each such $i$, we know by Proposition \ref{prop_constr}
\[\Dav_{\pm}(C_{m_i}\oplus  C_{m_{i-s}}) \ge \lfloor \log_2 m_{i}  \rfloor + \lfloor \log_2 m_{i-s}  \rfloor  + 2,\] and thus by repeatedly using Lemma \ref{prel_lem_lb} and possibly the lower bound for cyclic groups we get
a lower bound of $\sum_{i=1}^r \lfloor \log_2 m_i  \rfloor  + \lfloor r/2 \rfloor  + 1$
Now, the condition $\sum_{i=1}^r \{\log_2 m_i \} < \lfloor r/2 \rfloor + 1$ precisely guarantees that this lower bound indeed matches the upper bound.
\end{proof}
We point out that this result becomes vacuous if $r$ is too large, since then there are no $m_i$ such that the conditions can hold.
However, the result is meaningful for even $r$ up to $10$, and for $r$ equal to $3$ and $5$.

We phrase two special cases separately; for the latter an additional result is used in the proof.
Note that in the first result the condition $n\ge 4$ is essential; the situation for $n=3$ is different (see Theorem \ref{thm_el3}).

\begin{theorem}
\label{thm_lb}
Let $n \ge 4$ be an integer with $\{\log_2 n \} \ge \{\log_2 3 \}$.
If $r$ is a positive integer such that \[\{\log_2 n \} < \frac{\lfloor r/ 2 \rfloor + 1 }{r},\]
then $\Dav_{\pm}(C_n^r) = \lfloor r \log_2 n \rfloor +1$.
\end{theorem}
\begin{proof}
This is a direct application of Proposition \ref{prop_lb}.
\end{proof}
Again, the above result is vacuous for $r$ too large, see above for details. It covers in a certain sense the opposite end of the range of $\{\log_2 n \}$, relative to the one dealt with by $\eqref{eq_AGS}$. That inequality directly yields $\Dav_{\pm}(C_n^r) = \lfloor r \log_2 n \rfloor +1$ for $\{\log_2 n \} < 1/r$.

\begin{theorem}
\label{thm_33n}
Let $n \ge 2$ be a positive integer such that $ \{ \log_2 (3n)\} < 1 - \{\log_2 3 \}$ or $ \{ \log_2 (3n)\} \ge \{\log_2 3 \}$. Then
\[ \Dav_{\pm} (C_3 \oplus C_{3n}) = \lfloor \log_2 (9n)\rfloor + 1\]
\end{theorem}
\begin{proof}
By Theorem \ref{thm_standard} $\lfloor \log_2 (9n)\rfloor + 1$ is an upper bound; this is true for any $n$.
If $ \{ \log_2 (3n)\} < 1 - \{\log_2 3 \}$, then $\lfloor \log_2 (9n)\rfloor = \lfloor \log_2 (3)\rfloor + \lfloor \log_2 (3n)\rfloor $ and the claim follows by invoking the (standard) lower bound from Theorem \ref{thm_standard}.

If $ \{ \log_2 (3n)\} \ge \{\log_2 3 \}$, we can apply Proposition \ref{prop_lb} to get a lower bound of $\lfloor \log_2 (3)\rfloor + \lfloor \log_2 (3n)\rfloor + 2$, which equals $\lfloor \log_2 (9n)\rfloor + 1$.
\end{proof}

Results analogous to the one above could be established in other situations, for example for other groups of rank two, yet the situation here is particularly pleasant due to the special role of $3$ in Proposition \ref{prop_lb}.

\section{Improving the upper bound}

In the preceding sections we focused on lower bounds for $\Dav_{\pm}(G)$.
In this section we focus on bounding $\Dav_{\pm}(G)$ from above and show that
the general upper bound $\lfloor \log_2 |G| \rfloor +1$, sometimes can also be improved.

We start by recalling a special case of a result of Thangadurai \cite[Corollary 1.1]{thanga-paper}.
\begin{theorem}
\label{thm_el3}
Let $G$ be an elementary $3$-group. Then $\Dav_{\pm}(G)= \rank (G)+1$.
\end{theorem}
From this it directly follows that for elementary $3$-groups $\Dav_{\pm}(G)=\Dav_{\pm}^{\ast}(G)$, and that the upper bound  $\lfloor \rank (G) \log_2 3 \rfloor + 1 $ does not give the correct value for rank at least two.

This example of an improvement of the upper bound is however not quite satisfactory, in the sense that one could believe this is the only case where this happens due to the fact that this result actually is more naturally interpreted as one for a different type of weighted Davenport constant, and was also obtained originally in that context, which just happens to coincide with the plus-minus weighted Davenport constant in this case. We summarize this in the following subsection.

\subsection{Interlude on the fully weighted Davenport constant}

Let $G$ be a finite abelian group and let $n$ denote its exponent.
We refer to $\Dav_{\Z \setminus n \Z}(G)$, or equivalently $\Dav_{[1,n-1]}(G)$,  as the fully weighted Davenport constant.

The reason we call it fully weighted Davenport constant is that this set is
the largest one for which it makes sense to consider it; if the set of weights $A$ contains some multiple of the exponent then of course $\Dav_{A}(G)= 1$.

We determine this constant for every finite abelian groups; in \cite{thanga-paper} it was determined for elementary $p$-groups.

\begin{theorem}
\label{thm_full}
Let $G$ be a finite abelian group and let $n$ denote its exponent.
Then  $\Dav_{\Z \setminus n \Z}(G) = \rank_n(G)+1$.
\end{theorem}
\begin{proof}
We note that by the definition of $\rank_n(G)$ the group $G$ contains an independent subset of elements of order $n$ with cardinality $\rank_n(G)$.
The sequence formed by the elements of such a set cannot contain a $\Z \setminus n \Z$-weighted zero-subsum, since this would contradict the independence of the elements.

Conversely, let $S$ be a sequence of length $\rank_n(G) +1$.
We note that if $S$ contains an element $g$ of order $d$ with $d < n$, then
$d g$ is a  $\Z \setminus n \Z$-weighted zero-subsum.
Thus we may assume that each element in $S$ has order $n$.
Yet, by the definition of $\rank_n(G)$ this implies that the family formed by the elements occurring in $S$ (taking multiplicity into account) is not independent, yielding directly an  $\Z \setminus n \Z$-weighted zero-subsum.
\end{proof}

The proof of Theorem \ref{thm_el3} is now an immediate consequence.

\begin{proof}[Proof of Theorem \ref{thm_el3}]
Since as mentioned in Section \ref{sec_prel} the value of $\Dav_{A}(G)$
 depends only on the image of $A$ under the projection to  $\Z / \exp(G) \Z$, we have that for $G$ an elementary $3$-group the plus-minus weighted Davenport constant is equal to the fully weighted Davenport constant. The result follows by Theorem \ref{thm_full}.
\end{proof}

\subsection{Another example}

In view of the preceding discussion it seems desirable to have at least one example at hand where $\Dav_{\pm}(G)$ is less than the general upper bound, other than a case where it coincides with the fully weighted Davenport constant. We achieve this in the following result.

\begin{theorem} \label{thm_339}
We have
\[\Dav_{\pm}(C_3 \oplus C_3 \oplus C_9) = 6.\]
\end{theorem}
Note that the upper-bound provided by Theorem \ref{thm_standard} in this case is $1 + \lfloor \log_2 81 \rfloor = 7$. The rest of this section is dedicated to the proof of this result.

First, we note that Theorem \ref{thm_standard} yields $6$ as a lower bound. Let $S$ be a sequence over $G=C_3 \oplus C_3 \oplus C_9$ of length $6$. We need to show that $S$ has a plus-minus weighted zero-subsum, and assume for a contradiction it has none.

First, we show that $S$ contains only elements of order $9$.
Assume to the contrary that $S$ contains some element $h$ of order $3$.
Let $H$ denote the subgroup generated by $h$. Since $G/H$ is either isomorphic to $C_3^3$ or to $C_3 \oplus C_9$, we have $\Dav_{\pm}(G/H)\le 5$, where we use Theorem \ref{thm_standard}. Thus, $h^{-1}S$ has a plus-minus weighted subsum that is an element of $H$, so equal to $-h,h$ or $0$. In each case, we get a plus-minus weighted zero-subsum of $S$. Consequently, each element in $S$ has order $9$.

Let $f$ be some element occurring in $S$, recall it is of order $9$, and let $G = K \oplus \langle f \rangle$ where $K$ is a subgroup of $G$ isomorphic to $C_3^2$.

We consider the number of elements in $S$ that can be contained in $\langle f \rangle $. Since $\Dav_{\pm}(\langle f \rangle)= 4$, this number $t$ is at most $3$. We distinguish cases according to the value of $t$.

\medskip
\noindent
\textbf{Case $t=3$}: Denote the subsequence of $S$ of elements in $\langle f \rangle $  by $T$. Since $\Dav_{\pm}(C_9)=4$, the sequence $T$ is a sequence of maximal length without plus-minus weighted zero-subsum in $\langle f \rangle$, thus, since otherwise we could find a longer sequence, $\Sigma_{\pm}(T)\cup \{0\} = \langle f \rangle $; note that the inverse of each element in $\Sigma_{\pm}(T)$  is also an element of this set and the same is thus true for the complement of this set. Moreover, since $\Dav_{\pm}(K)=3$, it follows that $T^{-1}S$ has a plus-minus weighted subsum with sum in $\langle f \rangle$. Thus, we would get a plus-minus weighted zero-subsum of $S$, contradicting our assumption.

\medskip

\noindent
\textbf{Case $t=2$}: Since each element in $S$ has order $9$ and since we can replace, without restriction, an element by its inverse it follows that we can assume the two elements are $f,2f$ or $f,4f$. Yet, since $2 (4f)= -f$, and again changing the sign, we can in fact assume that the two elements are $f$ and $2f$. The set of plus-minus weighted subsums of these two elements augmented by $0$, is thus $\{-3f,-2f,-f,0,f,2f,3f\}$.

The set  $\Sigma_{\pm}((f (2f))^{-1}S) \cap \langle f\rangle$ must be disjoint from $\{-3f,-2f,-f,0,f,2f,3f\}$, as otherwise we would get a plus-minus weighted zero-subsum, so it is a subset of $\{-4f , 4f\}$.

We consider various subcases according to the structure of $R=(f (2f))^{-1}S$, more precisely of the structure of its image under the standard map  $\pi : K \oplus \langle f \rangle \to K$.

\smallskip
\noindent
\emph{Subcase 1.} Suppose $\pi(R)$ contains some element with multiplicity at least $3$. Say, $R$ contains a subsequence $(e+f_1)(e+f_2)(e+f_3)$ with $e\in K$ and $f_i \in \langle f \rangle$.

We then get the plus-minus weighted subsums $f_i - f_j$ for all distinct $i,j \in \{1,2,3\}$. By the observation made above they are all elements of $\{-4f,4f\}$. Yet, since $(f_3 -f_1) = (f_3-f_2) + (f_2 - f_1)$, this is impossible---the sum of two (possibly equal) elements of $\{-4f,4f\}$ cannot also be an element of this set.

\smallskip
\noindent
\emph{Subcase 2.} Suppose $\pi(R)$ contains two distinct elements with multiplicity at least $2$. Say, $R$ contains a subsequence $(e+f_1)(e+f_2)(e'+f_1')(e'+f_2')$ with $e,e'\in K$ and $f_i,f_i' \in \langle f \rangle$. We get that $(f_1-f_2)$ and $(f_1'-f_2')$ yet also $(f_1-f_2) + (f_1'-f_2')$, are all elements of $\{-4f,4f\}$, which is again impossible.

\smallskip
\noindent
\emph{Subcase 3.} Suppose $\pi(R)$ contains exactly one element with multiplicity two.
Since we can always replace an element by its inverse, we can in fact assume that $\pi(R) = e^2e'(-e-e')$ with independent $e,e'\in K$. Write $R= (e+f_1)(e+f_2)(e'+f')(-e-e'+f'')$ where the elements denote by $f$'s are in $\langle f \rangle$. It follows that $f_2 - f_1$, $f_1+f' + f''$, and $f_2+f'+f''$ are all elements of $\{-4f, 4f\}$. Yet since the last of the three, is the sum of the former two, this is impossible.

\smallskip
\noindent
\emph{Subcase 4.} Suppose $\pi(R)$ is squarefree. Again since we can replace an element by its inverse, we can assume $R = (e_1+f_1)(e_2+f_2)(-e_1-e_2+f_{1,2})(e_1-e_2+f_{1,2}')$ with $e_1,e_2$ independent elements of $K$ and $f_1,f_2, f_{1,2}, f_{1,2}' \in \langle f \rangle$.
We find that all the elements
$h_1= f_1 + f_2 + f_{1,2}$, $h_2 = -f_1 +f_2 +f_{1,2}'$, $h_3= -f_2 +f_{1,2}+f_{1,2}'$, and $h_4 = f_1 -f_{1,2}+f_{1,2}'$ are in $\{-4f,4f\}$.
And without restriction we may assume $h_1 = 4f$.

We consider first $h_1 + h_2 + 2h_3= 3(f_{1,2}+f_{1,2}')$;
this is an element of $4f + \{-4f, 4f\} + \{f,-f\}$ and an element of order $3$. Thus, $h_2= 4f$ and $h_3= -4f$.
Now, consider $h_1-h_2+h_4 = 3f_1$. Yet since we already got $h_1 = h_2$, while $h_4 \in  \{-4f, 4f\}$, this yields a contradiction.

\medskip
\noindent
\textbf{Case $t=1$}: We proceed similarly to the previous case and consider again subcases according to the structure of $R=f^{-1}S$ under the standard map  $\pi : K \oplus \langle f \rangle \to K$.

\smallskip
\noindent
\emph{Subcase 1.} $\pi(R)$ contains some element with multiplicity at least $4$;
denote the corresponding elements in $R$ by $e+f_i$ with $e \in K$ and $f_i \in \langle f \rangle$. Since each element in $S$ has order $9$, it follows that $f_i \in \{f,2f,4f,5f,7f,8f\}$. Yet then there have to be $f_i,f_j$ such that $f_i-f_j=-f$ yielding a plus-minus weighted zero-subsum.

\smallskip
\noindent
\emph{Subcase 2.} $\pi(R)$ contains some element with multiplicity $3$;
denote the corresponding elements in $R$ by $e+f_i$ with $e \in K$ and $f_i \in \langle f \rangle$.

Let us first assume the two remaining elements $g,h$ have the same image under $\pi$ or are inverses, then $g-h$ or $g+h$, resp., yields an element of $\langle f \rangle$. Assume without loss the former. The sequence $(g+h)(gh)^{-1}S$ then contains two elements in $\langle f \rangle$. While this sequence does not fulfil the conditions of the first subcase of the former case in a strict sense, we can observe that in the argument we only used the two elements in $\langle f \rangle$ and the three elements having the same image under $\pi$. Thus, that argument yields a plus-minus weighted zero-subsum of  $(g+h)(gh)^{-1}S$, and consequently of $S$.

So, assume $\pi(g)$ and $\pi(h)$ are independent. There is a choice of weights such that they yield a sum of the form $-e + f'$ for some $f' \in \langle f \rangle$. We then get that for $\{i,j,k\}= \{1,2,3\}$ the sequence $(f'+f_i)(f_j-f_k)f$ does not a have a plus-minus weighted zero-subsum. Thus $f_j-f_k,f'+f_i \in \{\pm 2f, \pm 4f\}$ for $\{i,j,k\}= \{1,2,3\}$. Yet this is impossible, for if
$f'+f_i \in \{\pm 2f, \pm 4f\}$ for each $i$, we certainly get a difference of two of these elements not in $\{\pm 2f, \pm 4f\}$.

\smallskip
\noindent
\emph{Subcase 3.}  $\pi(R)$ contains two distinct elements with multiplicity $2$.
We can assume that these two elements, call them $e_1,e_2$, are not inverses of each other and are thus independent; if they were inverses we are essentially reduced to Subcase 1. Moreover, the remaining element can also be assumed as pairwise independent of these two, otherwise we are essentially reduced to Subcase 2, and indeed we may assume it is equal to $-e_1-e_2$. In total we thus have $R= (e_1 + f_1)(e_1+f_1')(e_2 + f_2)(e_2+f_2') (-e_1 -e_2 + f')$ where the elements denoted by $f$'s are all in $\langle f \rangle$.

Now, since $(f_1-f_1')(f_2-f_2')f$ does not have a plus-minus weighted zero-subsum we may assume that $f_1'-f_1 = 2f$ and $f_2'-f_2=4f$. Moreover, we get that $f_1 +f_2 + f' \notin \{-f,0,f\}$ and likewise for $f_1' +f_2 + f'$ ,$f_1 +f_2' + f'$, and $f_1' +f_2' + f'$.

Yet, combining these observations, we get that $f_1 +f_2 + f' +m \notin \{-f,0,f\}$ for each $m \in \{0,2f, 4f, 6f\}$. Since $-\{0,2f, 4f, 6f\}+ \{-f,0,f\}= \langle f \rangle$, this is however impossible.

\smallskip
\noindent
\emph{Subcase 4.} $\pi(R)$ contains only one element with multiplicity $2$ (and all other elements have multiplicity at most one). Using our usual argument regarding inverses, and the knowledge of the preceding cases, we can assume that $\pi(R)= e_1^2 (e_1+e_2)e_2(-e_1+e_2)$, in other words $\pi(R)$ contains an element from each non-trivial cyclic subgroup of $K$, and one of them twice.

For the remaining argument it will however by advantageous to assume instead, which is however equivalent, that $\pi(R)= e_1(-e_1) (e_1+e_2)e_2(-e_1+e_2)$ and we write $R= (e_1+f_+)(-e_1+f_-)(e_1+e_2+a_1)(e_2+a_2)(-e_1+e_2+a_3)$ with $f$'s and $a$'s in $\langle f \rangle$.

It then follows that $f (f_++f_-)(a_1 +a_2+a_3)$ does not contain a plus-minus weighted zero-subsum. We observe that we may assume that $\{f_++f_-, a_1+a_2+a_3\}=\{2f, 4f \}$; to see this notice for example that we could consider the argument with respect to $-f$ instead of $f$, to put $f_++f_-$ in the desired set, and then consider if necessary the sequence where $(e_1+e_2+a_1),(e_2+a_2),(-e_1+e_2+a_3)$ are replaced by their inverses and consider $-e_2$ instead of $e_2$.

We now list various additional plus-minus weighted zero-subsums of $\pi(R)$ and the conditions the fact that they should not yield plus-minus weighted zero-subsums of $R$ yields for $f_+,f_-$ and $a_1,a_2,a_3$. In the end we will see that it is impossible to fulfill all these conditions.

The following are plus-minus weighted zero-subsums of $\pi(R)$; we put all elements in parenthesis to highlight the choice of weights:

\begin{itemize}
\item
$(e_1+e_2) - (e_2) -(e_1)$, $(e_1+e_2) - (e_2) +(-e_1)$, $(e_1+e_2) - (e_2)+(e_1) -(-e_1)$. The corresponding weighted subsums of $R$ are $a_1-a_2 -f_+$, $a_1-a_2 +f_-$,
$a_1-a_2 +f_+-f_-$.
\item $(e_2)-(-e_1+e_2) -(e_1)$, $(e_2)-(-e_1+e_2) +(-e_1)$, $(e_2)-(-e_1+e_2)+(e_1) -(-e_1)$. The corresponding weighted subsums of $R$ are $a_2-a_3 -f_+$, $a_2-a_3 +f_-$,
$a_2-a_3 +f_+-f_-$.
\item $(-e_1+e_2)-(e_1+e_2) -(e_1)$, $(-e_1+e_2)-(e_1+e_2) +(-e_1)$, $(-e_1+e_2)-(e_1+e_2)+(e_1) -(-e_1)$. The corresponding weighted subsums of $R$ are $a_3-a_1 -f_+$, $a_3-a_1 +f_-$,
$a_3-a_1 +f_+-f_-$.
\end{itemize}
Since $S$ is assumed to have no plus-minus weighted zero-subsum, it follows that the  weighted subsums of $R$ given above are not in $\{-f,0,f\}$. Setting $\mathcal{A}=\{a_1-a_2, a_2-a_3, a_3 -a_1 \}$ and $\mathcal{F} = \{-f_+,f_-, f_+-f_-\}$, this directly translates to
$(\mathcal{F} + \mathcal{A}) \cap \{-f,0,f\} = \emptyset$, and
further to
$\mathcal{A} \cap (-\mathcal{F}+ \{-f,0,f\}) = \emptyset$.
Of course, there would be other ways to rephrase this, but this last form is the one we found most convenient for the remaining argument.
 Now, we write $F=f_+ + f_-$ and $A= a_1 + a_2 + a_3$.
We observe that $-\mathcal{F} = \{f_+, -F +f_+, F -2f_+ \}$.

We know that $F=2f$ or $F=4f$ also recall that $f_+$ and $-F + f_+= -f_-$ have order $9$.

If $F=2f$ for $- \mathcal{F}$ this leaves the following possibilities (we keep the ordering of the elements for clarity):
$\{f, -f, 0 \}$, $\{ 4f, 2f, 3f \}$, $\{-2f, -4f, -3f\}$.
Consequently, for $ -\mathcal{F}  + \{ -f , 0 , f \} $ we have:
$\{-2f,-f,0,f,2f \}$, $\{-4f, f,2f, 3f, 4f\}$, $\{-4f, -3f, -2f, -f, 4f \}$.

If $F=4f$ for $- \mathcal{F}$ this leaves the following possibilities (we keep the ordering of the elements for clarity):
$\{2f, -2f, 0\}$, $(-4f, f, 3f)$, $(-f,4f,-3f)$.
Thus, for $ - \mathcal{F}  + \{ -f , 0 , f \} $ we have:
$\{-3f, -2f, -f, 0, f, 2f, 3f\}$, $\{-4f, -3f, 0, f, 2f, 3f, 4f\}$, $\{-4f, -3f, -2f, -f, 0, 3f, 4f\}$.

\smallskip
It now remains to decide if one of these sets can have non-empty intersection with the set $\mathcal{A}$ (for a suitable choice of parameters). We establish that this is indeed impossible.

First, we observe that the set $\mathcal{A}$ is of the form $\{c,d , -(c+d)\}$; yet some of these elements might be equal.
For the case  $F=4f$ only this is already sufficient to assert directly the impossibility by direct inspection; the complements of the three possible sets for $ - \mathcal{F}  + \{ -f , 0 , f \} $ listed above do not contain a subset of this form.
For the case $F=2f$ a finer analysis is needed. We start it with a preparatory observation.

We note that $a_3 -a_1 = A  -2a_1-a_2$ and  $a_2 - a_3 = -((a_3 - a_1) +(a_1-a_2))$, so
$\mathcal{A}= \{a,-(a+b),b\}$ with $a=a_1-a_2$ and $b = A - 2a_1 -a_2$.
It is now key to note that, for fixed $A$, the map from $C_9^2$ to itself given by $(a_1,a_2) \mapsto (a,b)$ with $a,b$ as just defined is not surjective.
More precisely,  after some transformations, we get that $(a,b)$ is the image of $(a_1, a_2)$ if and only if
$(a_1-a_2)= a$ and $-3a_2 = 2a+b-A$. Thus, for $(a,b)$ to be an image under the above map, we need that $2a+b-A$ is an element of order (at most) $3$.
Letting $a',b',A'$ be integers yielding $a,b,A$, resp., after multiplication with $f$, we can express this as
$2a'+b'-A' \equiv 0 \pmod{3}$ or $b'\equiv a'+A' \pmod{3}$.

After this preparation let us address the specific case $F=2f$ and (thus) $ A= 4f$ and, say, $A'=4$.

We consider each of the three sets individually, recall we need to find $a,b$ fulfilling  $b'\equiv a'+A' \pmod{3}$ such that $\{a,b,-(a+b)\}$ does not intersect the set.

For  $\{-2f,-f,0,f,2f \}$ we have: if $a=3f$, then $b=4f$; if $a=4f$, then $b=-4f$; if $a=-3f$, then $b=4f$; if $a=-4f$, then $b= \pm 3f$. In every case $-(a+b)$ is in the set.

For  $\{-4f, f,2f, 3f, 4f\}$ we have: if $a=-3f$, then $b= -2f$; if $a=-2f$, then $b=-f$; if $a=-f$, then $b=0$ or $b=3f$, if $a=0$, then $b=-2f$.
In every case $-(a+b)$ is in the set.

For $\{-4f, -3f, -2f, -f, 4f \}$ we have: if $a=0$, then $b=f$; if $a=f$, then $b=2f$; if $a=2f$, then $b=3f$ or $b=0$; if $a=3f$, then $b=f$.
In every case $-(a+b)$ is in the set.

So, there is no possibility to match all the conditions, establishing also in this final subcase a contradiction to the assumption that $S$, a sequence of length $6$, does not have a plus-minus weighted zero-subsum. This completes  the argument for $\Dav_{\pm}(C_3^2 \oplus C_9) \le 6$ and establishes the result.

\section{Concluding remarks and open problems}
\label{sec_rem}

In the following result we collect together some results on exact values of the plus-minus weighted
Davenport constant for certain groups; the specific selection of groups is done in such a way that
the value is given for all groups of order at most $100$ with one exception. In particular we establish the following result.

\begin{theorem}
\label{thm_100}
Let $G$ be a finite abelian group with $|G| \le 100$.
Then, $\Dav_{\pm}(G) = \lfloor \log_2 |G| \rfloor +1$, except for $G$ isomorphic to $C_3^2$, $C_3^3$, $C_3^2\oplus C_9$, where the values  are $3,4,6$, respectively, or $C_5 \oplus C_{15}$ where the value is either $6$ or $7$.
\end{theorem}

We give the proof of this result in a somewhat informal way, as it is not only a proof of this result but contains some additional information.

When below we say `the upper bound' we mean $\lfloor \log_2 |G| \rfloor + 1$ given by Theorem \ref{thm_standard} and by `the value' we mean the exact value of the plus-minus weighted Davenport constant (of some group at hand).

We recall that $\Dav_{\pm}(C_n)$ for each $n$ is known by Theorem \ref{thm_standard} and its value matches the upper bound.

The value of $\Dav_{\pm}(C_n^2)$ is known for $\{\log_2 n\} \ge \{\log_2 3\}$ by Theorem \ref{thm_lb} and for $\{\log_2 n\}< 1/2$ by Theorem \ref{thm_standard}; and the simple special cases for $n$ equal to $2$ and $3$.
The first value of $n$ not covered by this is $23$ the next are $46$ and $47$.
For all these $n$ except for $n=3$, the value matches the upper bound; for $n=3$ it is $3$.

The value of $\Dav_{\pm}(C_2 \oplus C_{2n})$ for $n \ge 2$ is known by Lemma \ref{prel_lem_lb2} and the result for cyclic groups and the value matches the upper bound.

The value of $\Dav_{\pm}(C_3 \oplus C_{3n})$ is known for $n \ge 2$ with $\{\log_2 3n \} < 1 -\{\log_2 3\}$ or $\{\log_2 3n \} \ge \{\log_2 3\}$ by Theorem \ref{thm_33n}. The first value of $n$ not covered is $15$ the next one $29$. The value matches the upper bound.

The value of $\Dav_{\pm}(C_4 \oplus C_{4n})$ for $n \ge 2$ is known by Lemma \ref{prel_lem_lb2} and the result for cyclic groups. The value matches the upper bound.

However, for  $\Dav_{\pm}(C_5 \oplus C_{5n})$ and $\Dav_{\pm}(C_7 \oplus C_{7n})$ the value is unknown already for $n=3$; for $n\le 2$ it is known by Lemma \ref{prel_lem_lb2} and the result for $C_n^2$ recalled above. Yet, of course the value is not unknown for each $n$ in these cases, and also for other groups of rank $2$ such results could be obtained.

We turn to groups of larger rank. If only the $2$-rank is greater than $2$, we can apply Lemma \ref{prel_lem_lb2} and in combination with the just established results for groups of rank $2$ we obtain the value in various case, in particular for all groups of the form $C_2 \oplus C_{2m} \oplus C_{2n}$ where $C_m \oplus C_n$ is a group where the value matches the upper bound. This includes all the groups mentioned above for which the value is known, except for $m=n=3$. However, for $C_2  \oplus C_6^2$ we can get the value by considering a subgroup of the form $C_2$ and again invoking Lemma \ref{prel_lem_lb2}; note that the value for $C_6^2$ also matches the upper bound.

Assume the $3$-rank is at least $3$. For elementary $3$-groups the result is known by Theorem \ref{thm_el3}. For $C_3\oplus C_3\oplus C_6$ we get the exact value by applying Lemma \ref{prel_lem_lb2}  with $C_3$ and $C_3\oplus C_6$ and the value matches the upper bound.
And, the value for $C_3\oplus C_3 \oplus C_9$ was determined in Theorem \ref{thm_339}.

In view of the above results, the only group of cardinality at most $100$ where the value thus remains unknown is $C_5 \oplus C_{15}$, that in this case it is $6$ or $7$ follows by Theorem \ref{thm_standard}. This completes the proof of Theorem \ref{thm_100}.

\subsection{Some problems} Of course, numerous questions still remain open, as is common for these types of constants. We mention some that we find of particular interest.

It could be interesting to investigate the value for  $C_5 \oplus C_{15}$, since it is the smallest group where the value of the plus-minus weighted Davenport constant is not known. Yet, some other groups seem even more interesting for a detailed investigation, specifically $C_{23}^2$ and possibly $C_3^3 \oplus C_{6}$ and $C_3^3 \oplus C_{9}$; the former on the grounds that knowing the value for it might give some evidence to formulate a conjecture for the value for $C_n^2$ in general, at least in case $n$ is prime, for the latter two since they are in some sense similar to the example $C_3^2 \oplus C_9$ where we were able to improve on the upper bound.

Based on Lemma \ref{prel_lem_lb} and Theorem \ref{thm_lb} one can easily construct groups where $\Dav_{\pm}(G) - \Dav_{\pm}^{\ast}(G)$ is arbitrarily large; for example for $G=C_7^{2r}$ the difference is at least $r$; that is a value essentially in the middle of the interval admissible by Corollary \ref{cor_davpm}.
For the dual question, $(\lfloor \log_2 |G| \rfloor  + 1) - \Dav_{\pm}(G)$, i.e., the difference to the upper bound from Theorem \ref{thm_standard}, the only type of groups for which this is know to become arbitrarily large are elementary $3$-groups, by Theorem \ref{thm_el3}. It is not clear to us if other examples exist or not.

We also do not yet have an example of a group where $\Dav_{\pm}(G)$ lies strictly between $\Dav_{\pm}^{\ast}(G)$ and $\lfloor \log_2 |G| \rfloor  + 1$, i.e. the upper and lower bound (relative to the optimal decomposition) in Theorem \ref{thm_standard}.
That such groups exists seems rather likely to us, as there seems no reason to expect that $\Dav_{\pm}(G) > \Dav_{\pm}^{\ast}(G)$, would always imply that $\Dav_{\pm}(G) = \lfloor \log_2 |G| \rfloor  + 1$ and that we do not have such an example seems merely due to the scarceness of examples of groups where the upper bound is \emph{known} not to be sharp.

\section*{Acknowledgements}

The authors would like to thank David Grynkiewicz for helpful discussions and remarks.

\end{document}